\documentclass{amsart}
\usepackage{amssymb}
\usepackage{amscd}
\usepackage{multicol}
\usepackage{MnSymbol}
\usepackage{tikz}
\usepackage[OT2,T1]{fontenc}
\DeclareSymbolFont{cyrletters}{OT2}{wncyr}{m}{n}
\DeclareMathSymbol{\Sha}{\mathalpha}{cyrletters}{"58}

\title[]
{On Lara Rodr\'iguez' full conjecture for double zeta values in function fields\\
}
\author{Ryotaro Harada}
\address{Graduate School of Mathematics, Nagoya University, 
Furo-cho, Chikusa-ku, Nagoya 464-8602 Japan }
\email{m15039r@math.nagoya-u.ac.jp}
\date{January 17, 2017.}

\newtheorem{thm}{Theorem}

\newtheorem{cor}[thm]{Corollary}
\newtheorem{prop}[thm]{Proposition}  

\theoremstyle{remark}
        
\theoremstyle{definition}

\newtheorem{rem}[thm]{Remark}

\newtheorem{conj}[thm]{Conjecture}    
\newtheorem{q}[thm]{Question}

\begin{document}
\bibliographystyle{amsalpha+}
\maketitle

\begin{abstract}
This paper discusses four formulae conjectured by J. A. Lara Rodr\'iguez on certain power series in function fields, which yield a 'harmonic product' formula for Thakur's double zeta values. We prove affirmatively the first two formulae. While we detect and correct errors in the last two formulae, and prove the corrected ones. 
%
\end{abstract}

\tableofcontents
\setcounter{section}{-1}
\section{Introduction}
Allegedly, in 1776, the double zeta values (multiple zeta values with depth 2) were firstly introduced by L. Euler in \cite{Eu} where he also described three types of relations for double zeta values with non-mathematical proofs and unconventional notations (they were reformulated with mathematical proofs and conventional modern notations in \cite{Ha}). It is said that the multiple zeta values were rediscovered after the silence of more than two centuries. 
In the last quarter century, it got known that they have connection to number theory (\cite{DG}, \cite{Z}), knot theory (\cite{LM}) and quantum field theory (\cite{BK}) and so on. Finding linear$\backslash$algebraic relations for multiple zeta values is one of our fundamental issues. Especially, the shuffle product formula and the harmonic product formula were discussed in detail in \cite{IKZ}.   

In 2004, the function field analogues of the multiple zeta values were invented by D. S. Thakur in \cite{Th}. He showed the existence of the 'harmonic product' formula for them in \cite{Th2010}. While in \cite{La}, J. A. Lara Rodr\'iguez conjectured its precise formulation in the case of depth $2$ with bounded weights. This conjecture contained five formulae. The first formula was proved by himself in \cite{La2012}. By using H. J. Chen's result in \cite{Ch}, we will prove affirmatively the second and third formulae in Theorem \ref{1} and \ref{2}. Whereas we detect and correct errors in the fourth and fifth formulae, and prove corrected ones in Theorem \ref{3} and \ref{4}.       
\section{Notations and Definitions}
\label{No}
\subsection{Notations}
We recall the following notation used in \cite{La2012}. 
\begin{itemize}
\setlength{\leftskip}{1.0cm}
\item[$q$] \quad  a power of a prime number $p$, $q=p^l$.  
\item[$\mathbb{F}_q$] \quad a finite field with $q$ elements.
\item[$A$] \quad the polynomial ring $\mathbb{F}_q[t]$.
\item[$A_{+}$] \quad the set of monic polynomials in $A$.   
\item[$A_{d^+}$]\quad the set of elements of $A_{+}$ of degree $d$.
\item[$\mathbb{F}_{q}(t)$] \quad the rational function field in the variable $t$.
\item[$\mathbb{F}_{q}((1/t))$]\quad the completion of $\mathbb{F}_{q}(t)$ at $\infty$.
\item[${\rm Int}(x)$]\quad $=\begin{cases}
				   0 & \text{if $x$ is not an integer,}\\
			         1 & \text{if $x$ is an integer. }
			        \end{cases}$ 
\end{itemize}
\subsection{Definition of multiple zeta values in $\mathbb{F}_{q}[t]$}
First we recall the power sums. For $s \in \mathbb{Z}$ and $d\in \mathbb{Z}_{\geq0}$, we write
\[
  S_d(s):=\sum_{a\in A_{d^+}}\frac{1}{a^s}\in\mathbb{F}_{q}(t).
\] 
For positive integers $s_1, s_2,\ldots, s_n\in\mathbb{Z}_{>0}$ and $d\in\mathbb{Z}_{\geq0}$, we put 
\[
  S_d(s_1, s_2, \ldots ,s_n) := S_d(s_1)\sum_{d>d_2>\cdots >d_n\geq0}S_{d_2}(s_2)\cdots S_{d_n}(s_n)\in\mathbb{F}_{q}(t).
\]
For $s \in \mathbb{Z}_{\geq 0}$, the {\it Carlitz zeta values}  are defined by
\[
  \zeta(s):=\sum_{a\in A_{+}}\frac{1}{a^s}\in \mathbb{F}_q((1/t)).
\]
Thakur generalized this definition to that of {\it multiple zeta values for} $\mathbb{F}_q[t]$ in \cite{Th}. For $s_1, s_2, \ldots, s_n\in\mathbb{Z}_{>0}$, 
\begin{align*}
  \zeta(s_1, s_2, \ldots, s_n)&:= \sum_{d_1>d_2> \cdots >d_n\geq0}S_{d_1}(s_1)\cdots S_{d_n}(s_n)\\
                                    &=\sum_{\substack{\deg a_1>\deg a_2>\cdots>\deg a_n\geq 0\\a_1, a_2, \ldots, a_n\in A_+}}\frac{1}{a^{s_1}_1\cdots a^{s_n}_n}\in \mathbb{F}_q((1/t)).
\end{align*}
For $a,b\in\mathbb{Z}_{>0}$, we define
\[
\Delta_d(a,b):=S_d(a)S_d(b)-S_d(a+b).
\]
H. J. Chen proved the following formula for the power sums in \cite{Ch} Theorem 3.1 and Remark 3.2.
\begin{prop}[Chen's formula]
For $r, s, d \in\mathbb{Z}_{\geq1}$, the following relation holds.
\[\Delta_d(r,s)=\sum_{\substack{ i+j=r+s\\q-1|j\\ i, j\geq 1}}\biggl\{ (-1)^{s-1}\binom{j-1}{s-1}+(-1)^{r-1}\binom{j-1}{r-1} \biggr\}S_d(i,j).
\]
Here we put $\binom{a}{b}=0$ for $a, b \in\mathbb{Z}_{\geq 0}$ with $a<b$. 
\end{prop}
We can determine the value of the binomial coefficients modulo $p$ by using Lucas's theorem (\cite{Lu} Section 3).
\begin{prop}[Lucas's Theorem]
Let $p$ be a prime number and $m, n \in \mathbb{Z}_{\geq0}$. Then we have 
\[ 
\binom{m}{n} \equiv \binom{m_0}{n_0}\cdots \binom{m_k}{n_k} \ {\rm mod}\ p
\] 
where $m=m_0+m_1p+\cdots +m_kp^k$ and $n=n_0+n_1p+\cdots +n_kp^k \ (m_i, n_i \in \{0, 1, \ldots, p-1\}$ for $i=0, 1, \ldots, k)$ are $p$-adic expansions of $m$ and $n$. 
\end{prop}
\section{Lara Rodr\'iguez' full conjecture and counter-examples}
\label{con}
Lara Rodr\'iguez conjectured several relations for Thakur's double zeta values in \cite{La}. We recall it in Section \ref{st}.  We detect some typos and errors in his formulae in Section \ref{rc}. 
\subsection{Statements}
\label{st}
The following is one of those conjectures which he called the full conjecture (\cite{La} Conjecture 2.8). It yields ''full'' descriptions of the 'harmonic product' formula for specific double zeta values (cf. \cite{La} Section 1). 
\begin{conj}[Lara Rodr\'iguez' full conjecture]
\label{lcon}
For $n, d \in \mathbb{Z}_{\geq1}$ and general $q$, we have
\begin{align}
\label{ii}\Delta_d(q^n+1, q^n)&={\rm Int}\biggl( \frac{2}{q} \biggr)S_d(2, 2q^n-1)\\
&\quad -\sum^{\frac{q^n-1}{q-1}}_{j=1}S_d\Bigl(3+(j-1)(q-1), 2q^n-2-(j-1)(q-1)\Bigr). \nonumber \\ 
\label{iii}\Delta_d(q^n-1, q^n+1)&=-\sum^{\frac{q^n+q-2}{q-1}}_{j=1}S_d\Bigl(2+(j-1)(q-1), 2q^n-2-(j-1)(q-1)\Bigr).\\
\label{iv}\Delta_d(q^{n-1}, q^n+1)&={\rm Int}\biggl(\frac{2}{q}\biggr)S_d\Bigl(2, q^n+q^{n-1}-1\Bigr) \\
						    &\ -\sum^{\frac{q^{n-1}-1}{q-1}}_{j=1}S_d\Bigl(3+(j-1)(q-1), q^n+q^{n-1}-2 + (j-1)(q-1)\Bigr).\nonumber \\ 
\intertext{For $0\leq i\leq n$, we have}
		\label{v}
\Delta_d(q^n+1, q^n+1-q^i)&={\rm Int}\biggl( \frac{2}{q} \biggr)S_d(2, 2q^n-q^i)\\
                                 &\quad -\sum^{\frac{q^n-q^i}{q-1}}_{j=1}S_d\Bigl(3+(j-1)(q-1), 2q^n-q^i-1-(j-1)(q-1)\Bigr)\nonumber \\
					&\quad +\sum^{\frac{q^n-q^i}{q-1}}_{j=\frac{q^n-q^i}{q-1}+1}S_d\Bigl(3+(j-1)(q-1), 2q^n-q^i-1-(j-1)(q-1)\Bigr).\nonumber
\end{align}
\end{conj}
\subsection{Remarks and Counter-examples}
\label{rc}
\begin{rem}
\label{re(1)}
Actually, in \cite{La} Conjecture 2.8 (2.8.1), Lara Rodr\'iguez conjectured one more relation
\[
\Delta_{d}(q^n, q^n-1)=-S_d(q^n, q^n-1).
\]
However he proved it in his later paper \cite{La2012} Theorem 6.3. 
\end{rem}

\begin{rem}
\label{re(2)}
The equation \eqref{ii} was stated as \cite{La} (2.8.2). In the case when $q=2$, this coincide with second formula in \cite{Th} Section 4.1.3.
The equation \eqref{ii} will be affirmatively proven in Theorem \ref{1}. 
\end{rem}

\begin{rem}
\label{re(3)}
The equation \eqref{iii} was stated as \cite{La} (2.8.3). In the case when $q=2$, this coincide with third formula in \cite{Th} Section 4.1.3. Again, the equation \eqref{iii} will be affirmatively proven in Theorem \ref{2}. 
\end{rem}

\begin{rem}
\label{re(4)}
The equation \eqref{iv} was stated as (2.8.4) in \cite{La} (in the case when $q=2$, this coincide with fourth formula in \cite{Th} Section 4.1.3). It looks that \eqref{iv} contains a typo, and furthermore it requires an additional term to correct it. 

Indeed it is quite curious to expect such an equality among the values with different weights (the sum of the first and the second components of double indices): In the right hand side of the equation \eqref{iv}, the first term is with weight $q^n+q^{n-1}+1$ while the summand of the second term is with weight $q^n+q^{n-1}+1+2(j-1)(q-1)$. 
In the case when $q=2$, $d=2$ and $n=3$, the equation \eqref{iv} claims 
\begin{align}
\label{lae}  
\Delta_2(4, 9)=S_2(2,11)-S_2(3,10)-S_2(4,11)-S_2(5,12),
\end{align}
while Chen's formula says
\begin{align}
\label{che}  
\Delta_2(4, 9)&=\sum_{\substack{i+j=13\\ i, j\geq1}}\biggl\{ \binom{j-1}{8}-\binom{j-1}{3} \biggr\}S_2(i,j) \\
                    &\equiv S_2(2,11)+S_2(3,10)+S_2(4,9)+S_2(5,8)+S_2(9,4)\ {\rm mod}\ 2.\nonumber
\end{align}
Therefore we must have
\begin{align}
\label{lach}
&S_2(2,11)+S_2(3,10)+S_2(4,11)+S_2(5,12)  \\
\quad &-S_2(2,11)-S_2(3,10)-S_2(4,9)-S_2(5,8)-S_2(9,4)\equiv 0\ {\rm mod}\ 2. \nonumber 
\end{align}
However,
\begin{align*}
  &S_2(4,11)+S_2(5,12)-S_2(4,9)-S_2(5,8)-S_2(9,4)\\
&\quad \equiv S_2(9,4)+S_2(5,12)+S_2(5,8)+S_2(4,11)+S_2(4,9)\ {\rm mod}\ 2 \\
  &\quad =S_2(9,4)+S_2(5)\Bigl(1+S_1(12)\Bigr)+S_2(5)\Bigl(1+S_1(8)\Bigr)\\
  &\qquad+S_2(4)\Bigl(1+S_1(11)\Bigr)+S_2(4)\Bigl(1+S_1(9)\Bigr)\\
  &\quad =S_2(9,4)+S_2(5)\Bigl(S_1(12)+S_1(8) \Bigr)+S_2(4)\Bigl(S_1(11)+S_1(9)\Bigr)
\end{align*}
Each term is calculated to be
\begin{align*}
  &S_2(9,4)\\
  &\quad \equiv \frac{\Bigl\{ \sum_{i=0}^{33}t^i+t^{32}+t^{31}+t^{30}+t^{26}+t^{25}+t^{22}+t^{17}+t^{16}+t^{15}+t^8+t^6+t^5+t^2 \Bigr\}}{t^{22}(t+1)^{19}(t^2+t+1)^9(t^2+1)^5}\\
  &\qquad\cdot(t+1)^6(t^2+1)^5\ {\rm mod}\ 2, \\
  &S_2(5)\Bigl(S_1(12)+S_1(8) \Bigr)\\
  &\quad\equiv  \frac{ \sum_{i=0}^{29}t^i+t^{26}+t^{25}+t^{23}+t^{21}+t^{20}+t^{19}+t^{16}+t^{13}+t^{12}+t^9+t^8+t^5}{t^{22}(t+1)^{19}(t^2+t+1)^9(t^2+1)^5}\\
   &\qquad \cdot(t^2+t+1)^4(t+1)^2\ {\rm mod}\ 2,\\
  &S_2(4)\Bigl(S_1(11)+S_1(9)\Bigr)\equiv \frac{\bigl\{ t^{12}+t^{5}+t^{4}+t^{3}+t^{2}+t+1 \bigr\}t^3(t^2+t+1)^8(t^2+1)^5}{t^{22}(t+1)^{19}(t^2+t+1)^{9}(t^2+1)^5}\ {\rm mod}\ 2.
\end{align*}
The degrees of numerators of $S_2(9,4)$, $S_2(5)\Bigl(S_1(12)+S_1(8) \Bigr)$ and $S_2(4)\Bigl(S_1(11)+S_1(9)\Bigr)$ are $49$, $39$ and $41$ respectively.  
Thus we find the degree of each numerator is different while they have the same denominators. Then it follows that $S_2(4,11)+S_2(5,12)-S_2(4,9)-S_2(5,8)-S_2(9,4)\nequiv 0\ {\rm mod}\ 2$ and this contradicts to \eqref{lach}. This gives the counter-example of \eqref{iv}. 

Therefore, we may correct \eqref{iv} as follows.
\begin{align}
\label{riv}
\Delta_d(q^{n-1}, q^n+1)&={\rm Int}\biggl(\frac{2}{q}\biggr)S_d\Bigl(2, q^n+q^{n-1}-1\Bigr) \\
						    &\ -\sum^{\frac{q^{n-1}-1}{q-1}}_{j=1}S_d\Bigl(3+(j-1)(q-1), q^n+q^{n-1}-2 - (j-1)(q-1)\Bigr).\nonumber
\end{align}
However, the above equation is not correct, due to a lack of an additional terms which is explained below:  
When $q=3, d=1$ and $n=3$, \eqref{riv} claims
\begin{align}
\label{vi} 
 \Delta_1(9, 28)=-S_1(3)-S_1(5)-S_1(7)-S_1(9).
\end{align} 
But according to Chen's formula, we have
\[
  \Delta_1(9, 28)=\sum_{\substack{i+j=37\\2|j\\i, j\in\mathbb{Z}_{\geq1}}}\biggl\{ -\binom{j-1}{27}+\binom{j-1}{8} \biggr\}S_1(i).
\] 
By Lucas's theorem, we find that the coefficient of $S_1(i)$'s vanish modulo $3$ except $-S_1(3), -S_1(5), -S_1(7), -S_1(9)$ and $S_1(19)$. That is,
\begin{align}
\label{vii}  
\Delta_1(9, 28)=-S_1(3)-S_1(5)-S_1(7)-S_1(9)+S_1(19).
\end{align}
By the definition of power sum, 
\[
  S_1(19)=\frac{1}{t^{19}}+\frac{1}{(t+1)^{19}}+\frac{1}{(t+2)^{19}}=\frac{t^{19}(t+2)^{19}+(t+1)^{19}(t+2)^{19}+t^{19}(t+1)^{19}  }{ t^{19}(t+1)^{19}(t+2)^{19}}.
\] 
The numerator of the right hand side has $2^{19}\equiv -1\ {\rm mod}\ 3$ as a constant term. Therefore $S_1(19)$ does not vanish modulo $3$. 
Thus \eqref{vi} contradicts to \eqref{vii}. So this suggests that we need additional terms to correct it.
In Theorem \ref{3}, we correct the equation \eqref{iv} as the equation \eqref{thm11} and prove it.
\end{rem}

\begin{rem}
\label{re(5)}
The equation \eqref{v} was stated as (2.8.5) in \cite{La}. Again, it looks that the equation \eqref{v} contains a typo because the summation of the third term in right hand side runs over the empty sum. We correct the equation \eqref{v} as the equation \eqref{thm12} and prove it in Theorem \ref{4}.
\end{rem}

%
%
%

\section{Main results}
\label{Ma}
In this section, we prove the first half of Lara Rodr\'iguez' full conjecture in Theorem \ref{1} and \ref{2}, we correct and prove the second half of the conjecture in Theorem \ref{3} and \ref{4}. Precisely in Theorem \ref{1} and \ref{2}, we show that the equations \eqref{ii} and \eqref{iii} hold. In Theorem \ref{3} and \ref{4}, we correct the equations \eqref{iv} and \eqref{v} as the equations \eqref{thm11} and \eqref{thm12} respectively and give their proofs.      
\begin{thm}
\label{1}
For $n\ and\ d \in \mathbb{Z}_{\geq1}$, the equation \eqref{ii} holds. 
\end{thm}

\begin{proof}

{\it Case 1} (the case when $q=2$). 
By Chen's formula for $q=2$,
\[
  \Delta_d(2^n+1, 2^n)=\sum_{\substack{i+j=2^{n+1}+1\\  i,j\geq 1}}\biggl\{ \binom{j-1}{2^n-1}+\binom{j-1}{2^n} \biggr\}S_d(i, j).
\]

When $j=2^{n+1}$, we have $\binom{2^{n+1}-1}{2^n-1}=\binom{2^{n+1}-1}{2^n}$. So we obtain
\begin{align*}
\binom{j-1}{2^n-1}+\binom{j-1}{2^n}=0
\end{align*}
for $j=2^{n+1}$.

When $1\leq j<2^n$, it is clear that 
\begin{align*}
\binom{j-1}{2^n-1}+\binom{j-1}{2^n}=0+0=0.
\end{align*}
 
When $2^n\leq j\leq 2^{n+1}-1$, let 
\[
  j-1=j_0+j_1\cdot 2+\cdots +j_{n}\cdot 2^{n}
\] 
be the $2$-adic expansion of $j-1$. 
The $2$-adic expansion of $2^{n}-1$ is given as follows
\begin{align*}
  2^n-1&= 1+2+2^2+\cdots +2^{n-1}.  
\end{align*}
By using Lucas's theorem,
\begin{align*}
  \binom{j-1}{2^n-1}&\equiv \prod_{k=0}^{n-1}\binom{j_k}{1}\binom{j_n}{0} \quad {\rm mod}\ 2,\\
  \binom{j-1}{2^n}&\equiv \prod_{k=0}^{n-1}\binom{j_k}{0}\binom{j_n}{1} \quad {\rm mod}\ 2.  
\end{align*}
Thus we obtain
\begin{align*}
\binom{j-1}{2^n-1}&\equiv 1 \ {\rm mod}\ 2\Leftrightarrow j-1=2^n-1\ {\rm or}\ 2^{n+1}-1,\\
\binom{j-1}{2^n}&\equiv 1 \ {\rm mod}\ 2\Leftrightarrow j_n=1 
\end{align*}
We always have $j_n=1$ for all $j$ with $2^n\leq j\leq 2^{n+1}-1$.
%
Therefore,
\[
\binom{j-1}{2^n-1}+\binom{j-1}{2^n} \equiv 
\begin{cases} 
 &1 \ {\rm mod}\ 2 \ \text{if $2^n\leq j\leq 2^{n+1}-1$},\\
 &0 \ {\rm mod}\ 2 \ \text{if $j\leq 2^n-1$\ or\ $j=2^{n+1}$}.	
\end{cases}
\]
Thus Chen's formula for $q=2$ becomes
\[
\Delta_d(2^n+1, 2^n)=\sum_{\substack{i+j=2^{n+1}+1\\ 2^n \leq j\leq 2^{n+1}-1}}S_d(i, j).
\]
Replacing $j$ with $2^{n+1}-j$, we have $2^n-1\leq 2^{n+1}-j\leq 2^{n+1}-1$ and thus $2\leq j\leq 2^n+1$.

Therefore
\begin{align*}		  
\Delta_d(2^n+1, 2^n)&=S_d(2, 2^{n+1}-1)+\sum^{2^n+1}_{j=3}S_d(j, 2^{n+1}+1-j)\\
			        &=S_d(2, 2^{n+1}-1)-\sum^{2^n-1}_{j=1}S_d(j+2, 2^{n+1}-1-j).
\end{align*}
So we obtain \eqref{ii}.

{\it Case 2} (the case when $q=p^l\neq2$). 
By Chen's formula, we have
\[
\Delta_d(q^n+1, q^n)=\sum_{\substack{ i+j=2q^n+1\\q-1|j}}\biggl\{ (-1)^{q^n-1}\binom{j-1}{q^n-1}+(-1)^{q^n}\binom{j-1}{q^n} \biggr\}S_d(i,j).
\] 
We obtain 
\[
(-1)^{q^n-1}\binom{j-1}{q^n-1}+(-1)^{q^n}\binom{j-1}{q^n}=\binom{j-1}{q^n-1}-\binom{j-1}{q^n}
\]
(we note that the above equation holds for $p=2$ because the characteristic is 2 in this case).

When $0<j \leq q^n-1$ with $q-1|j$, it is easily seen that 
\begin{align}\label{himeji}
\binom{j-1}{q^n-1}-\binom{j-1}{q^n}=0-0=0.
\end{align}

When $q^n+q-2\leq j\leq 2q^n-2$ with $q-1|j$, 
We put the $p$-adic expansions of $j-1$ and $q^n-1$ as follows
\begin{align*}
j-1&=j_0+j_1p+\cdots +j_{ln}p^{ln},\\
q^n-1&=p-1+(p-1)p+\cdots +(p-1)p^{ln-1}. 
\end{align*} 
Applying Lucas's theorem, we have
\begin{align*}
\binom{j-1}{q^n-1}&\equiv \prod^{ln-1}_{k=0}\binom{j_k}{p-1}\binom{j_{ln}}{0} \ {\rm mod}\ p,   \\
\binom{j-1}{q^n}&\equiv \prod^{ln-1}_{k=0}\binom{j_k}{0}\binom{j_{ln}}{1} \ {\rm mod}\ p.
\end{align*}
Thus it follows that
\begin{align*}
\binom{j-1}{q^n-1}&\nequiv 0 \ {\rm mod}\ p \Leftrightarrow j_k=p-1 \ (k\in\{ 0,1,\cdots ,ln-1\}),\\
\binom{j-1}{q^n}&\nequiv 0 \ {\rm mod}\ p \Leftrightarrow j_{ln}\neq 0.
\end{align*}
By the condition $q^n+q-1\leq j\leq 2q^n-2$, we have
\[
q^n+q-1=p^{ln}+p^l-1\leq j\leq p-2+(p-1)p+\cdots +(p-1)p^{ln-1}+p^{ln}=2q^n-2.
\]
So we always have $j_{ln}=1$. Then $\binom{j_{ln}}{1}=1$ for $j$ with $q^n+q-1\leq j\leq 2q^n-2$ and $q-1|j$.  It follows that 
\[
\binom{j-1}{q^n}\equiv 1 \ {\rm mod}\ p.
\]
If $j_k=p-1$ for all $k\in\{ 0,1,\ldots, ln-1\}$ we have $j-1=q^n-1+q^n=2q^n-1$ because we always have $j_{ln}=1$. This contradicts to the condition $q^n+q-2 \leq j \leq 2q^n-2$. Thus we have
\[
\binom{j-1}{q^n-1}\equiv 0 \ {\rm mod}\ p.
\]
for $j$ with $q^n+q-2 \leq j \leq 2q^n-2$ and $q-1|j$. 
Therefore
\begin{align}\label{akou}
\binom{j-1}{q^n-1}-\binom{j-1}{q^n}\equiv 1\ \ {\rm mod}\ p
\end{align}
for $j$ with $q^n+q-2\leq j\leq 2q^n-2$ and $q-1|j$. 

Therefore, by \eqref{himeji} and \eqref{akou}, we obtain
\[
  \binom{j-1}{q^n-1}-\binom{j-1}{q^n}\equiv
\begin{cases}
&0 \ \ {\rm mod}\ p\ \text{if $0<j\leq q^n-1$ with $q-1|j$,}\\
&1 \ \ {\rm mod}\ p\ \text{if $q^n+q-2\leq j\leq 2q^n-2$ with $q-1|j$.}
\end{cases}
\]
Then Chen's formula becomes
\begin{align*}
\Delta_d(q^n+1, q^n)=-\sum_{\substack{ i+j=2q^n+1\\q^n+q-2\leq j \leq 2q^n-2\\q-1|j} }S_d(i,j).
\end{align*}
Putting $i$ as $3+(j-1)(q-1)$ and $j$ as $2q^n-2-(j-1)(q-1)$, we have $q^n+q-2\leq 2q^n-2-(q-1)(j-1)\leq 2q^n-2$ and thus $1\leq j\leq \frac{q^n-1}{q-1}$. Therefore
\[
\Delta_d(q^n+1, q^n)=-\sum^{\frac{q^n-1}{q-1}}_{j=1}S_d(3+(j-1)(q-1), 2q^n-2-(j-1)(q-1)).
\]

Combining {\it Case 1} and {\it Case 2}, we obtain the equation \eqref{ii}.
\end{proof}

\begin{thm}
\label{2}
For $n \ and\ d \in \mathbb{Z}_{\geq1}$, the equation \eqref{iii} holds.
\end{thm}

\begin{proof}
By Chen's formula,
\[
  \Delta_d(q^n-1, q^n+1)=-\sum_{\substack{i+j=2q^n\\q-1|j} }(-1)^{q^n-1}\biggl\{ \binom{j-1}{q^n}+\binom{j-1}{q^n-2} \biggr\}S_d(i,j).
\]

We have 
\[
(-1)^{q^{n}-1}\biggl\{ \binom{j-1}{q^n}+\binom{j-1}{q^n-2} \biggr\}=\binom{j-1}{q^n}+\binom{j-1}{q^n-2}
\]
(we note that the above equation holds for $q=p^l$ with $p=2$ because the characteristic is $2$). 

When $j< q^n-1$ with $q-1|j$, it is clear that
\begin{align}\label{toyoka}
\binom{j-1}{q^n}+\binom{j-1}{q^n-2}=0+0=0.
\end{align}

When $q^n-1\leq j \leq q^n$ with $q-1|j$ , we have $j=q^n-1$ because $j$ satisfies $q-1|j$. Thus in this case we have 
\begin{align}\label{awaji}
\binom{j-1}{q^n}+\binom{j-1}{q^n-2}=\binom{q^n-2}{q^n}+\binom{q^n-2}{q^n-2}=0+1=1.
\end{align}

When $q^n<j\leq 2q^n-2$ with $q-1|j$, 
we set the $p$-adic expansion of $j-1$ as follows  
\[
  j-1=j_0+j_1p+\cdots j_{ln}p^{ln} \quad (j_k \in \{0, 1, \ldots, p-1\}).
\]
 By Lucas's theorem, we have 
\[
  \binom{j-1}{q^n}\equiv \prod_{k=0}^{ln-1}\binom{j_k}{0}\binom{j_{ln}}{1}\ {\rm mod}\ p.
\]  
So we have
\[
  \binom{j-1}{q^n}\equiv 1\ {\rm mod}\ p \Leftrightarrow j_{ln}=1.
\]
We always have $j_{ln}=1$ because $2q^n-2=(p-2)+(p-1)p+(p-1)p^2+\cdots +(p-1)p^{ln-1}+p^{ln}$ and $q^n=p^{ln}< j\leq2q^n-2$. Therefore
\[
  \binom{j-1}{q^n}\equiv 1\ {\rm mod}\ p. 
\]
Next we will prove $\binom{j-1}{q^n-2}\equiv 0\ {\rm mod}\ p$. Again using Lucas's theorem, we obtain 
\[
  \binom{j-1}{q^n-2}\equiv \binom{j_0}{p-2}\prod^{ln-1}_{k=1}\binom{j_k}{p-1}\binom{1}{0} 
\] 
by $q^n-2=p-2+(p-1)p+\cdots+(p-1)p^{ln-1}+0\cdot p^{ln}$. Thus 
\[
  \binom{j-1}{q^n-2}\nequiv 0\ {\rm mod}\ p \Leftrightarrow p-2\leq j_0\leq p-1\ {\rm and}\ j_k=p-1\ \text{for all $k\in \{1, 2, \ldots, ln-1\}$}. 
\]
If $j_0=p-2$ and $j_k=p-1$ for all $k\in \{1, 2, \ldots ln-1\}$, we have $j-1=p-2+(p-1)p+\cdots +(p-1)p^{ln-1}+p^{ln}=2q^n-2$. However, $j=2q^n-1$ is not divisible by $q-1$. If $j_0=p-1$ and $j_k=p-1$ for all $k\in \{1, 2, \ldots ln-1\}$, we have $j-1=p-1+(p-1)p+\cdots +(p-1)p^{ln-1}+p^{ln}=2q^n-1$. But $j=2q^n$ is not divisible by $q-1$. Thus we always have
\[
\binom{j-1}{q^n-2}\equiv 0\ {\rm mod}\ p.
\] 
Therefore we have
\begin{align}\label{oita} 
 \binom{j-1}{q^n}+\binom{j-1}{q^n-2}  \equiv 1 \ {\rm mod}\ p
\end{align}
for $j$ with $q^n-1\leq j\leq 2q^n-2$ and $q-1|j$. 

By \eqref{toyoka}, \eqref{awaji} and \eqref{oita}, we obtain
\begin{align}
\binom{j-1}{q^n}+\binom{j-1}{q^n-2} \equiv
\begin{cases} 
 & 0\ \ {\rm mod}\ p\quad \text{if $j<q^n-1$ with $q-1|j$} \\
 & 1\ \ {\rm mod}\ p\quad \text{if $q^n-1\leq j\leq 2q^n-2$ with $q-1|j$}.\\
\end{cases}
\end{align}
Therefore Chen's formula becomes
\[ 
  \Delta_d(q^n-1, q^n+1)=-\sum_{\substack{i+j=2q^n\\q^n-1\leq j \leq 2q^n-2\\q-1|j}}S_d(i,j).
\]
Replacing $j$ with $2q^n-2-(j-1)(q-1)$, we have $q^n-1\leq 2q^n-2-(j-1)(q-1) \leq 2q^n-2$ and thus $1\leq j\leq \frac{q^n+q-2}{q-1}$. Therefore
\[
\Delta_d(q^n-1, q^n+1)=-\sum^{\frac{q^n+q-2}{q-1}}_{j=1}S_d\Bigl(2+(j-1)(q-1), 2q^n-2-(j-1)(q-1)\Bigr).
\]

Combining {\it Case 1} and {\it Case 2}, we obtain the equation \eqref{iii}. 
%
%
\end{proof}

As we saw in Remark \ref{re(4)}, the equation \eqref{iv} contains errors. We correct them as follows.  
\begin{thm}
\label{3}
For $d, n\in\mathbb{Z}_{\geq 1}$, we have
\begin{align}
\label{thm11}
\Delta_d(q^{n-1}, q^n+1)&={\rm Int}\biggl(\frac{2}{q}\biggr)S_d(2, q^n+q^{n-1}-1) \\
						    &\quad -\sum^{\frac{q^{n-1}-1}{q-1}}_{j=1}S_d\Bigl(3+(j-1)(q-1), q^n+q^{n-1}-2-(j-1)(q-1)\Bigr)\nonumber \\
						    &\quad +S_d(2q^{n-1}+1, q^n-q^{n-1}).\nonumber 
\end{align}
\end{thm}
\begin{proof} 
By $1\leq j\leq \frac{q^{n-1}-1}{q-1}$, we have 
\[
  q^n+q-2\leq q^n+q^{n-1}-2-(j-1)(q-1)\leq q^{n}+q^{n-1}-2.
\]
By replacing $q^n+q^{n-1}-2-(j-1)(q-1)$ by $j$, we see that it is enough to prove
\begin{align}
\label{okayama}  
\Delta_{d}(q^{n-1}, q^n+1)=&{\rm Int}\Bigl(\frac{2}{q}\Bigr)S_d(2, q^n+q^{n-1}-1)-\sum_{\substack{q^n+q-2\leq j\leq q^n+q^{n-1}-2\\ q-1|j\\ i+j=q^n+q^{n-1}+1}}S_d(i,j)\\
&+S_d(2q^{n-1}+1, q^n-q^{n-1}),\nonumber
\end{align} 
which is a reformulation of \eqref{thm11}.

We note that Chen's formula says 
\begin{align}
\label{gifu}  
\Delta_d(q^{n-1}, q^{n}+1)=\sum_{\substack{i+j=q^n+q^{n-1}+1\\q-1|j\\i,j\in\mathbb{Z}_{\geq1}}}\biggl\{ (-1)^{q^n}\binom{j-1}{q^n}+(-1)^{q^{n-1}-1}\binom{j-1}{q^{n-1}-1} \biggr\}S_d(i,j).
\end{align} 

{\it Case 1} (the case when $q=2$). 
The equation \eqref{gifu} becomes
\[
  \Delta_d(2^{n-1}, 2^n+1)=\sum_{i+j=2^{n-1}+2^n+1}\biggl\{\binom{j-1}{2^n}+\binom{j-1}{2^{n-1}-1}\biggr\}S_d(i,j).
\]

When $0\leq j-1<2^{n-1}-1$, it is easily seen that
\begin{align}\label{okinawa}
\binom{j-1}{2^n}+\binom{j-1}{2^{n-1}-1}=0+0=0.  
\end{align}

When $2^{n-1}-1 \leq j-1<2^n$, it is clear that 
\[
\binom{j-1}{2^n}=0. 
\]
By Lucas's theorem, we have
\[
  \binom{j-1}{2^{n-1}-1}\equiv \prod^{n-2}_{k=0}\binom{j_k}{1}\binom{j_{n-1}}{0} \ {\rm mod} \ 2
\]
where $j-1=j_0+j_12+\cdot +j_{n-1}2^{n-1}$ is the 2-adic expansion of $j-1$.
Therefore
\[ 
\binom{j-1}{2^{n-1}-1}\equiv 1\ {\rm mod}\ 2 \Leftrightarrow j_k=1\ \text{for all}\ k\in \{ 0, 1,\ldots, n-2\} 
\]
Thus
\[
\binom{j-1}{2^{n-1}-1}\equiv 
\begin{cases}
 &1\ \text{if $j=2^{n-1}$ or $2^n$},\\
 &0\ \text{if $2^{n-1}<j<2^n$}.
\end{cases} 
\]
Therefore
\begin{align}\label{hokkaido}
\binom{j-1}{2^n-1}+\binom{j-1}{2^{n-1}-1}\equiv
\begin{cases}
 &1\ \ {\rm mod}\  2\ \text{if $j=2^{n-1}$\ or\ $2^n$ },\\
 &0\ \ {\rm mod}\  2\ \text{if $2^{n-1}<j<2^n$}.
\end{cases}
\end{align}

When $2^n\leq j-1\leq 2^n+2^{n-1}-1$, it is clear that 
by Lucas's theorem, 
\[
  \binom{j-1}{2^{n}}\equiv \prod^{n-1}_{k=0}\binom{j_k}{0}\binom{j_{n}}{1} \ {\rm mod} \ 2
\]
Then we have 
\[
  \binom{j-1}{2^{n}}\equiv 1 \ {\rm mod} \ 2 \Leftrightarrow j_n=1.
\]
By the condition $2^n\leq j-1\leq 2^n+2^{n-1}-1$, we have $j_n=1$. So we always have
\[
  \binom{j-1}{2^n}\equiv 1\ {\rm mod}\ 2\ \text{for all $j$ with $2^n\leq j-1\leq 2^n+2^{n-1}-1$}.
\]
On the other hand,   
\[
  \binom{j-1}{2^{n-1}-1}\equiv 1\ {\rm mod}\ 2 \Leftrightarrow j_k=1
\]
for all $k\in \{ 0, 1,\ldots, n-2\}$.
In this case, we have $j-1=2^n+2^{n-1}-1+j_{n-1}2^{n-1}$. By the condition  $2^n \leq j-1\leq 2^n+2^{n-1}-1$, it must be $j-1=2^n+2^{n-1}-1$. Then we obtain
\begin{align}\label{chiba}
  \binom{j-1}{2^n}+\binom{j-1}{2^{n-1}-1}\equiv
     \begin{cases} 
&1\ \ {\rm mod}\ 2\ \text{if $2^n+1\leq j \leq 2^n+2^{n-1}-1$},\\
&0\ \ {\rm mod}\ 2\ \text{if $j=2^n+2^{n-1}$.}
      \end{cases}
\end{align}

Therefore by \eqref{okinawa}, \eqref{hokkaido} and \eqref{chiba},
\begin{align*}
  \binom{j-1}{2^n}+\binom{j-1}{2^{n-1}-1}\equiv
     \begin{cases} &1\ \ {\rm mod}\ 2\ \text{if $j=2^{n-1}$ or $2^n\leq j \leq 2^n+2^{n-1}-1$},\\ 
                       &0\ \ {\rm mod}\ 2\ \text{if $1\leq j<2^{n-1}$, $2^{n-1}<j<2^n$ or $j=2^n+2^{n-1}$.}
      \end{cases}
\end{align*}
It concludes the following equation
\begin{align*}
  \Delta_d(2^{n-1}, 2^n+1)=
                                    &S_d(2, 2^n+2^{n-1}-1)+\sum_{\substack{2^n\leq j\leq2^n+2^{n-1}-2\\ i+j=2^n+2^{n-1}+1}}S_d(i,j)+S_d(2^n+1, 2^{n-1}).
\end{align*}
Thus we get the equation \eqref{okayama}.

{\it Case 2} (the case when $q=p^l\neq 2$ ). 
On coefficients of \eqref{gifu}, we have 
\[
   (-1)^{q^n}\binom{j-1}{q^n}+(-1)^{q^{n-1}-1}\binom{j-1}{q^{n-1}-1}=-\binom{j-1}{q^n}+\binom{j-1}{q^{n-1}-1}
\]
(we again note that above equation holds for $q=p^l$ with $p=2$ because the characteristic is $2$).

When $0< j-1\leq q^{n-1}-2$ with $q-1|j$, it is easily seen that 
\begin{align}
\label{shiga}
  -\binom{j-1}{q^n}+\binom{j-1}{q^{n-1}-1}=-0+0=0.
\end{align} 

When $q^{n-1}+q-3\leq j-1\leq q^n-2 $ with $q-1|j$, it is clear that
\begin{align*}
\binom{j-1}{q^n}=0.
\end{align*}
Set the $p$-adic expansions of $j-1$ and $q^{n-1}-1$ as follows
\begin{align*}
  j-1&=j_0+j_1p+\cdots +j_{ln-1}p^{ln-1},\\
  q^{n-1}-1&=p-1+(p-1)p+\cdots (p-1)p^{l(n-1)-1}.
\end{align*}
By Lucas's theorem, 
\[
  \binom{j-1}{q^{n-1}-1}\equiv \prod^{l(n-1)-1}_{k=0}\binom{j_k}{p-1}\prod^{ln-1}_{r=l(n-1)}\binom{j_r}{0} \ {\rm mod}\ p. 
\]
Therefore we obtain
\begin{align*}
 \binom{j-1}{q^{n-1}-1}\nequiv 0\ {\rm mod}\ p \Leftrightarrow j_k=p-1 
\end{align*}
where for all $k\in \{ 0,1,\ldots, l(n-1)-1 \}$. In this case we have
\begin{align*}
  j-1&=q^{n-1}-1+j_{l(n-1)}p^{l(n-1)}+\cdots +j_{ln-1}p^{ln-1}\\
      &=q^{n-1}-1+p^{l(n-1)}(j_{l(n-1)}+j_{l(n-1)+1}p+\cdots +j_{ln-1}p^{l-1})\\
      &=q^{n-1}-1+q^{n-1}(j_{l(n-1)}+j_{l(n-1)+1}p+\cdots +j_{ln-1}p^{l-1}).
\end{align*}
Then we have $0\leq j_{l(n-1)}+j_{l(n-1)+1}p+\cdots +j_{ln-1}p^{l-1}\leq q-1$.
By the condition $q-1|j$, we have $j_{l(n-1)}+j_{l(n-1)+1}p+\cdots +j_{ln-1}p^{l-1}=q-2$. Thus $j=q^{n-1}+q^{n-1}(q-2)=q^{n}-q^{n-1}$, and in this case we have $\binom{j-1}{q^{n-1}-1}\equiv 1\ {\rm mod}\ p$. 
So we obtain
\[
\binom{j-1}{q^{n-1}-1}\equiv 
\begin{cases}
 &1\ \ {\rm mod}\ p\quad \text{if $j=q^n-q^{n-1}$},\\ 
 &0\ \ {\rm mod}\ p\quad \text{if $j\neq q^n-q^{n-1}$}. 
\end{cases}
\]
Therefore 
\begin{align}\label{akita}
-\binom{j-1}{q^{n}}+\binom{j-1}{q^{n-1}-1}\equiv 
\begin{cases}
&1\ \ {\rm mod}\ p\quad \text{if $j=q^n-q^{n-1}$},\\ 
&0\ \ {\rm mod}\ p\quad \text{if $j\neq q^n-q^{n-1}$}. 
\end{cases}
\end{align}

When $ q^n+q-3\leq j-1\leq q^n+q^{n-1}-3$ with $q-1|j$, put the $p$-adic expansion of $j-1$ as follows
\[
  j-1=j_0+j_1p+\cdots +j_{ln}p^{ln}.
\]
By Lucas's theorem, 
\[
  \binom{j-1}{q^n}\equiv \prod^{ln-1}_{k=0}\binom{j_k}{0}\binom{j_{ln}}{1}\ {\rm mod}\ p. 
\]
The condition $q^n+q-3\leq j-1\leq q^n+q^{n-1}-3$ implies that $j_{ln}=1$. Therefore, we have
\begin{align*}
\binom{j-1}{q^n}\equiv 1 \ {\rm mod}\ p
\end{align*}
for all $j$ with $q^n+q-3\leq j-1\leq q^n+q^{n-1}-3$ and $q-1|j$.
Again by Lucas's theorem, we have
\[
  \binom{j-1}{q^{n-1}-1}\equiv \prod^{l(n-1)-1}_{k=0}\binom{j_k}{p-1}\prod^{ln}_{r=l(n-1)}\binom{j_r}{0}\ {\rm mod}\ p.
\]
If $\binom{j-1}{q^{n-1}-1}\nequiv 0\ {\rm mod}\ p$, then $j_k=p-1$ for all $k\in\{0,1,\ldots, l(n-1)-1\}$. It means $j-1\equiv q^{n-1}-1\ \ {\rm mod}\ q^{n-1}.$ 
But it contradicts to $q^n+q-3\leq j-1\leq q^n+q^{n-1}-3$ 
(we note that here we use $q\neq2$).
Therefore
\begin{align*}
\binom{j-1}{q^{n-1}-1}\equiv 0\ {\rm mod}\ p
\end{align*}     
for all $j$ with $q^n+q-3 \leq j-1\leq q^n+q^{n-1}-3$ and $q-1|j$.
Then we have 
\begin{align}\label{kanagawa}
-\binom{j-1}{q^n}+\binom{j-1}{q^{n-1}-1}=-1+0=-1.
\end{align}
By \eqref{shiga}, \eqref{akita} and \eqref{kanagawa}, we have 
\begin{align*}
 (-1)^{q^n}\binom{j-1}{q^n}&+(-1)^{q^{n-1}-1}\binom{j-1}{q^{n-1}-1}\\
&\quad \equiv
\begin{cases} 
 &\ 0\ \ {\rm mod}\ p\quad \text{if $j\leq q^n-1$ with $q-1|j$ and $j\neq q^n-q^{n-1}$},\nonumber \\
 &\ 1\ \ {\rm mod}\ p\quad \text{if $j=q^n-q^{n-1}$},\\
 &-1\ \ {\rm mod}\ p\quad \text{if $q^n+q-2\leq j\leq q^n+q^{n-1}-2$ with $q-1|j$}.\\
\end{cases}
\end{align*}
Therefore we obtain \eqref{okayama} by \eqref{gifu}.
%

Combining the {\it Case 1} and {\it Case 2}, we obtain the equation \eqref{okayama}. Therefore the equation \eqref{thm11} follows.
\end{proof}

We correct the equation \eqref{v} as follows.
\begin{thm}
\label{4}
We set $d, n\in\mathbb{Z}_{\geq 1}$.  For $0\leq s\leq n$, the following equation holds
\begin{align}
\label{thm12}
\Delta_d(q^n+1, q^n-q^s+1)&={\rm Int}\biggl( \frac{2}{q} \biggr)S_d(2, 2q^n-q^s)\\
                                 &\quad -\sum^{\frac{q^n-q^s}{q-1}}_{j=1}S_d\Bigl(3+(j-1)(q-1), 2q^n-q^s-1-(j-1)(q-1)\Bigr)\nonumber \\
					&\quad +\sum^{\frac{q^n-1}{q-1}}_{j=\frac{q^n-q^s}{q-1}+1}S_d\Bigl(3+(j-1)(q-1), 2q^n-q^s-1-(j-1)(q-1)\Bigr).\nonumber
\end{align}
\end{thm}
We remark that when $s=0$ (resp. $s=n$), the third term (resp. the second term) of the right hand side of \eqref{thm12} means the empty sum.  
We note that in the case when $s=0$, it recovers \eqref{ii}.
\begin{proof}
We have $q^n+q-q^s-1\leq 2q^n-q^s-1-(j-1)(q-1)\leq 2q^n-q^s-1$ when $1\leq j\leq \frac{q^n-1}{q-1}$. Replacing $2q^n-q^s-1-(j-1)(q-1)$ with $j$, we see it is enough to prove 
\begin{align}
\label{tokyo}  
\Delta_d(q^n+1, q^n+1-q^s)=&{\rm Int}\biggl( \frac{2}{q} \biggr)S_d(2, 2q^n-q^s)\\
                                      &-\sum_{\substack{q^n+q-2 \leq j\leq 2q^n-q^s-1\\ i+j=2q^n-q^s+2\\ q-1|j }}S_d(i,j)+\sum_{\substack{q^n-q^s+q-1\leq j\leq q^n-1 \\i+j=2q^n-q^s+2\\q-1|j }}S_d(i,j).\nonumber
\end{align}

{\it Case 1} (the case when $q=2$). 
Chen's formula becomes
\[
  \Delta(2^n+1, 2^n-2^s+1)=\sum_{i+j=2^{n+1}-2^s+2}\biggl\{ \binom{j-1}{2^n-2^s}+\binom{j-1}{2^n} \biggr\}S_d(i,j).
\]

When $0\leq j-1 <2^n-2^s$, it is easily seen that 
\begin{align}\label{nigata}
\binom{j-1}{2^n-2^s}+\binom{j-1}{2^n}=0+0=0.
\end{align}

When $ 2^n-2^s \leq j-1<2^n$, it is clear that
\[ 
\binom{j-1}{2^n}=0.
\] 
We put the $2$-adic expansion of $j-1$ by
\[
  j-1=j_0+j_1\cdot 2+\cdots +j_{n-1}\cdot 2^{n-1}. 
\]
By Lucas's theorem, 
\[
  \binom{j-1}{2^n-2^s}\equiv \prod^{s-1}_{k=0}\binom{j_k}{0}\prod_{r=s}^{n-1}\binom{j_r}{1} \ {\rm mod}\ 2.  
\]
Then we have
\[
\binom{j-1}{2^n-2^s}\equiv 1 \ {\rm mod}\ 2 \Leftrightarrow j_r=1
\]
for all $r\in \{s, s+1, \ldots, n-1\}$.
And if $j_r=1$ for all $r\in \{s, s+1, \ldots, n-1\}$, we have
\begin{align*}
  j-1&=j_0+j_1\cdot 2 +\cdots +j_{s-1}\cdot 2^{s-1}+2^s+2^{s+1}+\cdots +2^{n-1}\\
      &=j_0+j_1\cdot 2+\cdots+j_{s-1}\cdot 2^{s-1}+2^n-2^s. 
\end{align*}
So, when $2^n-2^s\leq j-1<2^n$, we always have 
\begin{align*}
  \binom{j-1}{2^n-2^s}\equiv 1\ {\rm mod}\ 2.
\end{align*}
Therefore
\begin{align}\label{nara}
\binom{j-1}{2^n-2^s}+\binom{j-1}{2^n}\equiv1\ \ {\rm mod}\ 2
\end{align}
for all $j$ with $2^n-2^s+1\leq j <2^n+1$.

When $2^n \leq j-1\leq 2^{n+1}-2^s$, put the $2$-adic expansion of $j-1$ by 
\[
j-1=j_0+j_1\cdot 2 +\cdots +j_n\cdot 2^n. 
\]
By Lucas's theorem, 
\begin{align*}
  &\binom{j-1}{2^n}\equiv \prod_{k=0}^{n-1}\binom{j_k}{0}\binom{j_n}{1} \ {\rm mod}\ 2\\
  &\binom{j-1}{2^n-2^s}\equiv \prod_{k=0}^{s-1}\binom{j_k}{0}\prod^{n-1}_{r=s}\binom{j_r}{1}\binom{j_n}{0}    \ {\rm mod}\ 2
\end{align*}
Then we obtain 
\begin{align*}
  &\binom{j-1}{2^n}\equiv 1 \ {\rm mod}\ 2 \Leftrightarrow j_n=1,\\
  &\binom{j-1}{2^n-2^s}\equiv 1 \ {\rm mod}\ 2 \Leftrightarrow j_r=1\ \text{for all $r \in \{s, s+1,\ldots, n-1\}.$}
\end{align*}
We always have $j_n=1$ because $2^n \leq j-1 \leq 2^{n+1}-2^s$ and $2^n\leq 2^{n+1}-2^s\leq 2^{n+1}-1$. So 
\begin{align*}
\binom{j-1}{2^n}\equiv 1 \ {\rm mod}\ 2
\end{align*}
for all $j$ with $ 2^n\leq j-1\leq 2^{n+1}-2^{s}$.
While if $j_r=1$ for all $s \leq r\leq n-1$, 
\begin{align*}
  j-1&=j_0+j_1\cdot 2+\cdots +j_{s-1}\cdot2^{s-1}+2^n-2^s+2^n\\
      &=j_0+j_1\cdot 2+\cdots +j_{s-1}\cdot2^{s-1}+2^{n+1}-2^s
\end{align*}  
because $j_n=1$. Thus we have $j_0=j_1=\cdots=j_{s-1}=0$ by the condition  $2^n \leq j-1 \leq 2^{n+1}-2^s $ and hence $j-1=2^{n+1}-2^s$. So 
\begin{align*}
  \binom{j-1}{2^n-2^s}\equiv 1\ {\rm mod}\ 2  \Leftrightarrow j=2^{n+1}-2^s+1.
\end{align*}
Then we have
\begin{align}\label{wakayama}
  \binom{j-1}{2^n}+\binom{j-1}{2^n-2^s}\equiv
\begin{cases}
  &1\ \ {\rm mod}\ 2\quad \text{if $2^n+1\leq j \leq 2^{n+1}-2^s$. }\\
  &0\ \ {\rm mod}\ 2\quad \text{if $j=2^{n+1}-2^s+1$.}
\end{cases}
\end{align}

Therefore by \eqref{nigata}, \eqref{nara} and \eqref{wakayama},
\[
  \binom{j-1}{2^n}+\binom{j-1}{2^n-2^s}\equiv
\begin{cases}
  &1\ \ {\rm mod}\ 2\quad \text{if $2^n-2^s+1\leq j \leq 2^{n+1}-2^s$. }\\
  &0\ \ {\rm mod}\ 2\quad \text{if $1\leq j<2^n-2^s$ or $j=2^{n+1}-2^s+1$.}
\end{cases}
\]
It concludes that we obtain
\begin{align}
\Delta_d(2^n+1, 2^n-2^s+1)&=\sum_{\substack{2^n-2^s+1\leq j\leq 2^{n+1}-2^s \\i+j=2^{n+1}-2^s+2 }}S_d(i,j)\\
                                    &=S_d(2, 2^{n+1}-2^s)+\sum_{\substack{2^n-2^s+1\leq j\leq 2^{n+1}-2^s-1\\i+j=2^{n+1}-2^s+2 }}S_d(i,j).\nonumber
\end{align}
This corresponds to \eqref{tokyo} for $q=2$.   

{\it Case 2} (the case when $q=p^l\neq 2$ ). 
Chen's formula says
\begin{align*}
  \Delta_d(q^n+1, q^n-q^s+1)&=\sum_{\substack{i+j=2q^n-q^s+2\\q-1|j }}\biggl\{(-1)^{q^n-q^s}\binom{j-1}{q^n-q^s}+(-1)^{q^n}\binom{j-1}{q^n}  \biggr\}S_d(i,j).
\end{align*}
We have
\[
  (-1)^{q^n-q^s}\binom{j-1}{q^n-q^s}+(-1)^{q^n}\binom{j-1}{q^n}=\binom{j-1}{q^n-q^s}-\binom{j-1}{q^n}
\]
(we note that the above equation holds for $q=p^l$ with $p=2$ because the characteristic is $2$).

When $s=0$, we have \eqref{thm12} because it is equivalent to \eqref{ii}. 

When $s=n$, Chen's formula becomes
\[
 \Delta_d(q^n+1, 1)=\sum_{\substack{i+j=q^n+2\\q-1|j}}\biggl\{\binom{j-1}{0}-\binom{j-1}{q^n} \biggr\}S_d(i,j).
\] 
It is easily seen that 
\[
\binom{j-1}{0}-\binom{j-1}{q^n}=1-0=1 
\]
for all $j$ with $q-1\leq j\leq q^n-1$ and $q-1|j$. Thus we have
\[
  \Delta_d(q^n+1, 1)=\sum_{\substack{q-1\leq j \leq q^n-1 \\i+j=q^n+2 \\q-1|j}}S_d(i,j).
\] 
Hence we get \eqref{tokyo} and therefore the equation \eqref{thm12} holds in this case.

So we may assume that $1\leq s\leq n-1$.


%

When $0 < j-1 \leq q^n-q^s-1$ with $q-1|j$ , it is easily seen that 
\begin{align}\label{saitama}
  \binom{j-1}{q^n-q^s}-\binom{j-1}{q^n}=0-0=0.
\end{align}

When $ q^n-q^s+q-2 \leq j-1\leq q^n-2$ with $q-1|j$, it is clear that $\binom{j-1}{q^n}=0$. In this case, we put the $p$-adic expansion of $j-1$ and $q^n-q^s$ by
\begin{align*}
  j-1&=j_0+j_1\cdot p+\cdots +j_{ln-1}p^{ln-1},\\
q^n-q^s&=(q^n-1)-(q^s-1)=(p-1)p^{ls}+(p-1)p^{ls+1}+\cdots +(p-1)p^{ls-1}.
\end{align*}
 Applying Lucas's theorem, 
\[
  \binom{j-1}{q^n-q^s}\equiv \prod_{k=0}^{ls-1}\binom{j_k}{0}\prod_{m=ls}^{ln-1}\binom{j_m}{p-1}\ {\rm mod}\ p. 
\]
Then we have
\[   
  \binom{j-1}{q^n-q^s}\equiv 1\ {\rm mod}\ p \Leftrightarrow j_m=p-1\ \text{for all $m\in \{ls, ls+1, \ldots, ln-1\}$.} 
\]
If $j_m=p-1$ for all $m \in \{ls, ls+1, \ldots, ln-1\}$, we obtain 
\begin{align*}
  j-1&=j_0+j_1p+\cdots +j_{ls-1}p^{ls-1}+(p-1)p^{ls}+\cdots +(p-1)p^{ln-1}\\ 
      &=j_0+j_1p+\cdots +j_{ls-1}p^{ls-1}+q^n-q^s.
\end{align*}
Since we have $0 \leq j_0+j_1p+\cdots +j_{ls-1}p^{ls-1}\leq q^s-1$, 
we have $\binom{j-1}{q^n-q^s}\equiv 1\ {\rm mod}\ p$ for all $j$ with $q^n-q^s+q-2\leq j-1\leq q^n-2$.
Therefore we have
\begin{align}\label{tokushima}
\binom{j-1}{q^n-q^s}-\binom{j-1}{q^n}\equiv 1\ \ {\rm mod}\ p
\end{align}
for all $j$ with $q^n-q^s+q-2\leq j-1\leq q^n-2$ and $q-1|j$. 

When $q^n+q-3\leq j-1\leq 2q^n-q^s-2$ with $q-1|j$, we may put the $p$-adic expansion of $j-1$ by
\[
  j-1=j_0+j_1p+\cdots +j_{ln}p^{ln}.
\]
By using Lucas's theorem, 
\[
  \binom{j-1}{q^n}\equiv \prod_{k=0}^{ln-1}\binom{j_k}{0}\binom{j_{ln}}{1}\ {\rm mod}\ p.
\]
This shows 
\[
\binom{j-1}{q^n}\equiv 1\ {\rm mod}\ p \Leftrightarrow j_{ln}=1.
\]
 As $j-1$ satisfies $q^n+q-3=p^{ln}+p^{l}-3 \leq j-1\leq p^{ln}+p^{ln}-p^{ls}-2=2q^n-q^s-2$, we always have $j_{ln}=1$.
Thus we obtain
\[
  \binom{j-1}{q^n}\equiv 1\ {\rm mod}\ p
\]
for all $j$ with $q^n+q-3\leq j-1\leq 2q^n-q^s-2$ and $q-1|j$.
Whereas by Lucas's theorem,
\[
  \binom{j-1}{q^n-q^s}\equiv \prod_{k=0}^{ls-1}\binom{j_k}{0}\prod_{m=ls}^{ln-1}\binom{j_m}{p-1}\binom{j_{ln}}{0}\ {\rm mod}\ p
\]
and therefore 
\[
\binom{j-1}{q^n-q^s}\nequiv 0\ {\rm mod}\ p\Leftrightarrow j_m=p-1\ \text{for all $m\in \{ls, ls+1, \dots, ln-1\}$ }.
\]
If $\binom{j-1}{q^n-q^s}\nequiv 0\ {\rm mod}\ p$,
then $j_m=p-1$ for all $m\in \{ls, ls+1, \dots, ln-1\}$. It means $j-1=j_0+j_1p+\cdots +j_{ls-1}p^{ls-1}+q^{n}-q^s+q^n$ and thus $2q^n-q^s\leq j-1$. However this $j$ does not satisfy $q^n+q-3\leq j-1\leq 2q^n-q^s-2$. Thus we must have
\[
\binom{j-1}{q^n-q^s}\equiv 0\ {\rm mod}\ p
\]
for all $j$ with $q^n+q-3\leq j-1\leq 2q^n-q^s-2$ and $q-1|j$.

So it follows that
\begin{align}\label{toyama}
\binom{j-1}{q^n-q^s}-\binom{j-1}{q^n}\equiv -1\ {\rm mod}\ p
\end{align}
for all $j$ with $q^n+q-3\leq j-1\leq 2q^n-q^s-2$ and $q-1|j$.

Therefore by \eqref{saitama}, \eqref{tokushima} and \eqref{toyama}, we obtain 
\begin{align*}
  \binom{j-1}{q^n-q^s}-\binom{j-1}{q^n}\equiv 
\begin{cases}
&\ 0\ \ {\rm mod}\ p\ \text{if $1< j\leq q^n-q^s$ with $q-1|j$,}\\
&\ 1\ \ {\rm mod}\ p\ \text{if $q^n-q^s+q-1\leq j \leq q^n-1$ with $q-1|j$, }\\
&-1\ \ {\rm mod}\ p\ \text{if $q^n+q-2 \leq j\leq 2q^n-q^s-1$ with $q-1|j$. }
\end{cases}
\end{align*}
So \eqref{tokyo} holds in this case by Chen's formula.

Combining the {\it Case 1} and the {\it Case 2}, we have \eqref{tokyo}. Therefore \eqref{thm12} follows.  
\end{proof}

Summing all of the equation \eqref{ii}, \eqref{iii}, \eqref{thm11} and \eqref{thm12} over $d$, we obtain the following corollary.
\begin{cor}
The following 'harmonic product' formula holds for double zeta values in function fields:
\begin{align}
\zeta(q^n+1)\zeta(q^n)=&\zeta(q^n+1, q^n)+\zeta(q^n, q^n+1)+\zeta(2q^n+1)+{\rm Int}\biggl( \frac{2}{q} \biggr)\zeta(2, 2q^n-1)\\
& -\sum^{\frac{q^n-1}{q-1}}_{j=1}\zeta\Bigl(3+(j-1)(q-1), 2q^n-2-(j-1)(q-1)\Bigr),\nonumber
\end{align}
\begin{align}
\zeta(q^n-1)\zeta(q^n+1)=&\zeta(q^n-1, q^n+1)+\zeta(q^n+1, q^n-1)+\zeta(2q^n)\\
& -\sum^{\frac{q^n+q-2}{q-1}}_{j=1}\zeta\Bigl(2+(j-1)(q-1), 2q^n-2-(j-1)(q-1)\Bigr),\nonumber
\end{align}
\begin{align}
\zeta(q^{n-1})\zeta(q^n+1)=&\zeta_d(q^{n-1}, q^n+1)+\zeta(q^n+1, q^{n-1})+\zeta(q^n+q^{n-1}+1)\\
& +{\rm Int}\biggl(\frac{2}{q}\biggr)\zeta(2, q^n+q^{n-1}-1)\nonumber \\
						    & -\sum^{\frac{q^{n-1}-1}{q-1}}_{j=1}\zeta\Bigl(3+(j-1)(q-1), q^n+q^{n-1}-2-(j-1)(q-1)\Bigr)\nonumber \\
						    & +\zeta(2q^{n-1}+1, q^n-q^{n-1}),\nonumber
\end{align}
and for $0\leq s\leq n$,
\begin{align}
\zeta(q^n+1)\zeta(q^n+1-q^s)=&\zeta(q^n+1, q^n+1-q^s)+\zeta(q^n+1-q^s, q^n+1)+\zeta(2q^n+2-q^s)\\
&+{\rm Int}\biggl( \frac{2}{q} \biggr)\zeta(2, 2q^n-q^s)\nonumber \\
&-\sum^{\frac{q^n-q^s}{q-1}}_{j=1}\zeta\Bigl(3+(j-1)(q-1), 2q^n-q^s-1-(j-1)(q-1)\Bigr)\nonumber \\
					&+\sum^{\frac{q^n-1}{q-1}}_{j=\frac{q^n-q^s}{q-1}+1}\zeta\Bigl(3+(j-1)(q-1), 2q^n-q^s-1-(j-1)(q-1)\Bigr).\nonumber
\end{align}
\end{cor}


\section*{Acknowledgments}
The author is deeply grateful to Professor H. Furusho for guiding him towards this topic. This paper could not have been written without his continuous encouragements. The author also gratefully acknowledges Professor J. A. Lara Rodr\'iguez for answering several questions which the author posed regarding \cite{La} Conjecture 2.8.

\end{document}